\documentclass[12pt]{amsart}
\usepackage{amscd,amssymb, latexsym}
\usepackage{amsmath,amsthm,amstext, amsbsy, amssymb}
\usepackage{epsfig,graphics,graphicx,mathrsfs,color}
\DeclareMathOperator{\dist}{dist}

\DeclareMathOperator{\SF}{{\mathcal S\mathcal F}}
\let\Re\xRe

\newtheorem{theorem}{Theorem}
\newtheorem*{theoremKor}{Korevaar's theorem}

\newtheorem{lemma}[theorem]{Lemma}
\newtheorem{corollary}[theorem]{Corollary}
\newtheorem{proposition}[theorem]{Proposition}

\newtheorem{remark}[theorem]{Remark}
\newtheorem*{conjectureChui}{Chui's Conjecture}

\theoremstyle{remark}

\newtheorem{question}{Question}

\makeatletter
\@addtoreset{equation}{section}
\makeatother

\begin{document}

\title[Chui's conjecture in 
Bergman spaces]{Chui's conjecture in Bergman spaces}

\author[E.~Abakumov, A.~Borichev, K.~Fedorovskiy]{Evgeny Abakumov, Alexander Borichev, Konstantin Fedorovskiy}

\begin{abstract}
We solve Chui's conjecture on the simplest fractions (i.e., sums of Cauchy kernels with unit coefficients) 
in weighted (Hilbert) Bergman spaces. Namely, for a wide class of weights, we prove that for every $N$, the simplest fractions with $N$ poles on the unit circle have minimal norm if and only if the poles are equispaced on the circle. We find sharp asymptotics of these norms. Furthermore, we describe the closure of the simplest fractions in weighted Bergman spaces,  using an $L^2$ version of Thompson's theorem on dominated approximation by simplest fractions.
\end{abstract}

\thanks{This work was carried out in the framework of the project 19-11-00058
by the Russian Science Foundation.}

\address{
\hskip -\parindent Evgeny Abakumov:
\newline \indent Universit\'e Gustave Eiffel, Marne-la-Vall\'ee, France
\newline \indent {\tt evgueni.abakoumov@univ-eiffel.fr}
\smallskip
\newline \noindent Alexander Borichev:
\newline \indent Aix--Marseille University, CNRS, Centrale Marseille, I2M, France
\newline \indent {\tt alexander.borichev@math.cnrs.fr}
\smallskip
\newline \noindent Konstantin Fedorovskiy:
\newline \indent Bauman Moscow State Technical University, Moscow, Russia,
\newline \indent Saint Petersburg State University, St.~Petersburg, Russia
\newline \indent {\tt kfedorovs@yandex.ru}
}

\maketitle

\section{Introduction}
The starting point of our research is the following question: How to put $N$ point charges on the unit circle $\mathbb T$ of the complex plane $\mathbb C$ in order to minimize the average strength of the corresponding electrostatic field in the unit disk $\mathbb D$, assuming forces inversely proportional to the distance? C.~K.~Chui \cite{Chu1971monthly} conjectured in 1971 that this average strength is minimal when the charges are equispaced on $\mathbb T$, and, surprisingly, this very natural and elementary conjecture is still open.

This and related questions call for the study of approximation properties of so-called simplest fractions.

We mean by {\it a simplest fraction}
(the term {\it simple partial fraction} is also used in the literature) a
rational function $r$ in the complex variable $z$ having the form
$$
r(z)=\sum_{0\le k<N}\frac1{z-a_k},
$$
where $N$ is a positive integer, and $a_k$, $0\le k<N$, are points in $\mathbb C$. Note that the simplest fraction $r(z)$ can be
represented as the logarithmic derivative of the polynomial
$\prod_{0\le k<N}(z-a_k)$; alternatively, the function $r(z)$ can be thought of as the Cauchy transform
of the sum of the Dirac measures of mass one at the
points $a_0,\ldots,a_{N-1}$. Another interpretation of the simplest
fraction $r$ is that the value $r(z)$ represents the complex conjugation of
the electrostatic field at the point $z$, caused by charges
placed at the points $a_0,\ldots,a_{N-1}$, assuming forces inversely proportional to the distance.

The simplest fractions are an interesting and important object
in various topics of contemporary analysis. One can mention here, for example, that in 2006 J.~M.~Anderson and V.~Eiderman \cite{AndEid2006annals} 
solved a long-standing problem of Macintyre--Fuchs 
describing the growth of the Hausdorff content of the level sets of simplest fractions.
For a recent survey of numerous results on simplest fractions 
see \cite{3auth}.

G.~R.~Mac Lane \cite{ML} initiated in 1949 the study of approximation by polynomials with restriction on the location of their zeros. 
One says that a set $E$ is {\it a polynomial approximation set relative to a domain $G$} if every zero-free holomorphic function $f$ on $G$ can be approximated uniformly on compact subsets of $G$ by polynomials having zeros only on $E$. 
Mac Lane showed that for every bounded (simply connected) 
Jordan domain $\Omega \subset \mathbb C$ with rectifiable boundary, $\partial\Omega$ is a polynomial approximation 
set relative to $\Omega$. 
Later, M.~Thompson \cite{MT}, Chui \cite{Chu1971TAMS}, and Z.~Rubinstein and E.~B.~Saff \cite{RS} strengthened this result by considering bounded polynomial approximation in the unit disc.

Furthermore, J.~Korevaar \cite{Kor1964ann} considered Mac Lane's problem in a more general setting and related it to approximation by simplest fractions, which is 
one of the main objects of consideration in the present paper. The main result of Korevaar reads as follows:




\begin{theoremKor}
Let $G$ be a bounded simply connected domain in $\mathbb C$, and let $E\subset \mathbb C$ be 
such that $E\cap G=\emptyset$. The following four statements are equivalent:

\smallskip
{\rm 1}\textup) The set $E$ is a polynomial approximation set relative to $G$.

\smallskip
{\rm 2}\textup) For every point $w\in E$, the function $z\mapsto (z-w)^{-1}$ can be
approximated locally uniformly in $G$ by polynomials having zeros only on
$E$.

\smallskip
{\rm 3}\textup) There exists a system of finite families
$\{a_{N,k}\colon 0\le k<N\}$ of points in $E$,  such that
$$
\sum_{0\le k<N}\frac1{z-a_{N,k}}\to0
$$
locally uniformly in $G$ as $N\to\infty$.

\smallskip
{\rm 4}\textup) The set $\mathop{\rm clos}E$ separates the plane,
and $G$
belongs to a bounded connected component of the set $\mathbb
C\setminus\mathop{\rm clos}E$.
\end{theoremKor}

Thus, 
the possibility of approximation 
in the sense of the first assertion of the theorem
is equivalent to the possibility of approximation of the zero function by
simplest fractions with poles on $E$.

For a given set $E\subset\mathbb C$, we consider the family of all simplest fractions with poles on the set $E$:
$$
\SF(E)=\bigg\{\sum_{0\le k<N}\frac1{z-a_{N,k}}: N\ge 1,\,a_{N,k}\in E,\, 0\le k<N\bigg\}.
$$

As a corollary of Korevaar's theorem, one has the following result about
approximation of general holomorphic functions (not necessarily zero--free ones) by
simplest fractions with restrictions on the poles. Let $G$
be a bounded simply connected domain in $\mathbb C$, and let $K$ be a compact
subset of $G$ having connected complement. Then the family $\SF(\partial G)$ is
dense in the space $A(K)$ consisting of all continuous functions on
$K$ which are holomorphic in the interior 
of $K$.
Notice that Korevaar's results were recently extended by 
P.~A.~Borodin \cite{Bor2016sbm}.

In the beginning of 1970-s, Chui \cite{Chu1971monthly} considered yet
another problem related with the approximation by the simplest fractions with poles lying on the unit circle. 
He was
interested in the question whether the set $\SF=\SF(\mathbb T)$ is dense in
the Bergman space $A^1=A^1(\mathbb D)$ 
consisting of all functions holomorphic and integrable in $\mathbb D$, and 
endowed with the usual $L^1$-norm (with respect to normalized planar Lebesgue measure $m_2$ on the unit disk, 
$dm_2(z)=\pi^{-1}\,dxdy$, $z=x+iy$).

In connection with this question Chui formulated the
following 
conjecture.

\begin{conjectureChui}
For any positive integer $N$, and for any family of points $\{a_{k}\}_{0\le k<N}$ on the unit circle, we have 
$$
\bigg\|\sum_{0\le k<N}\frac1{z-a_{k}}\bigg\|_{L^1(\mathbb D)}\ge
\bigg\|\sum_{0\le k<N}\frac1{z-e^{2\pi ik/N}}\bigg\|_{L^1(\mathbb D)}.
$$
\end{conjectureChui}

For $N\ge 1$ we denote 
$$
\Psi_N(z)=\sum_{0\le k<N}\frac1{z-e^{2\pi ik/N}}.
$$

It can be easily verified (see \cite{Chu1971monthly}) that
\begin{equation}\label{eq:chui-newman}
\|\Psi_N\|_{L^1(\mathbb D)}\ge  C
\end{equation}
for some absolute constant $C>0$. Thus, Chui's conjecture would imply that the set $\SF$ is not dense in $A^1$.


The next year after the publication of Chui's conjecture, D.~J.~Newman
\cite{New1972monthly} proved that the set $\SF$ is not dense in $A^1$. More precisely, he established that 
$$
\bigg\|\sum_{0\le k<N}\frac1{z-a_k}\bigg\|_{L^1(\mathbb D)}\ge \dfrac\pi{18}
$$
for any collection $\{a_k\}_{0\le k<N}$ of points on the unit circle.


Next, Chui studied in \cite{Chu1969PAMS}  approximation by simplest fractions in
Jordan domains in the complex plane in the Bers spaces, that is, in the
weighted $L^1$-spaces with weights $\lambda_D^{2-q}$, $0<q<\infty$, where
$\lambda_D$ is the Poincar\'e metric for the domain $D$ under consideration. It
follows from estimate \eqref{eq:chui-newman} that the set $\SF$ is not dense
in the Bers spaces in $\mathbb D$ for every $1<q\le 2$. 
It is proved in \cite{Chu1969PAMS} that for any Jordan domain $D$ and for every $q>2$, the simplest
fractions with poles on the boundary of $D$ are dense in the respective Bers
space in $D$. The results of \cite{Chu1969PAMS} were later extended in \cite{ChuShen}.

Despite considerable progress in our knowledge of simplest fractions properties, including approximation ones, the
original question posed by Chui remains open. In this paper, we resolve 
a version of 
Chui's conjecture in the context of weighted Bergman spaces of square integrable functions, that is in the Hilbert space setting.


Throughout the paper we use the following notation: for positive $A$ and $B$, 
$A \lesssim B $ means that there is a positive numerical constant $C$ such that
$A \le CB $, while $A \gtrsim B $ means that $B \lesssim A $, and
$A \asymp B $ means that both $A \lesssim B $ and
$B \lesssim A $.

\section{Main results}

Let us recall that for $\alpha>-1$ the (standard) weighted Bergman space $A^2_{\alpha}=A^2_{\alpha}(\mathbb D)$ consists of all
functions $f$ holomorphic in $\mathbb D$ for which the norm
$\|f\|_{\alpha}$ is finite, where
$$
\|f\|^2_{\alpha}=
(\alpha+1)\int_{\mathbb D} |f(z)|^2\,(1-|z|^2)^{\alpha}\,dm_2(z).
$$
We refer the reader to the book 
\cite{HedKorZhu2000book} where one can find a thorough
exposition 
of the theory of standard weighted Bergman spaces.

More generally, if $g$ is an integrable positive function on the interval $[0,1]$, we consider the corresponding
weighted Bergman space
$$
A^2_{(g)}=\Bigl\{f\in\mathop{\rm Hol}(\mathbb D):\|f\|^2_{(g)}=
\kappa_g\int_{\mathbb D} |f(z)|^2\,g(1-|z|^2)\,dm_2(z)<\infty\Bigr\},
$$
where $\kappa_g=(\int_0^1g(t)\,dt)^{-1}$ is the normalization constant. 
It can be verified directly that the fractions $(z-\lambda)^{-1}$, $\lambda\in\mathbb T$, belong to 
$A^2_{(g)}$ if and only if 
\begin{equation}\label{eq:g1}
\int_0\frac{g(s)}{s}\,ds<\infty.
\end{equation}
Recall that we denote by $\SF$ the set of all simplest fractions with poles on $\mathbb T$. 
For every $\alpha>0$ we have 
$\SF\subset A^2_{\alpha}$, 
and for every $\alpha\in(-1,0]$ we have $\SF \cap A^2_{\alpha}=\emptyset$. 
So, in what follows we suppose that $\alpha>0$. 

First, we establish an analogue of Chui's conjecture for a wide class of Bergman weighted spaces, namely, we prove that  for  $N$ point masses on the unit circle,  the norm of the corresponding Cauchy transform is the smallest if and only if these point masses  are equispaced on $\mathbb T$.

\begin{theorem}\label {thm:lb} Let $g\not\equiv 0$ be a concave non-decreasing function on $[0,1]$ 
satisfying \eqref{eq:g1} and such that $g(0)=0$. 
Then for every integer $N\ge 1$ and for every family of points $\{a_{k}\}_{0\le k<N}$ on the unit circle we have
$$
\bigg\|\sum_{0\le k<N}\frac1{z-a_k}\bigg\|_{(g)}\ge \bigg\|\sum_{0\le k<N}\frac1{z-e^{2\pi ik/N}}\bigg\|_{(g)}=
\big\|\Psi_N \big\|_{(g)}.
$$
Furthermore, if $\{a_{k}\}_{0\le k<N}$ are points on the unit circle such that 
$$
\bigg\|\sum_{0\le k<N}\frac1{z-a_k}\bigg\|_{(g)}=\big\|\Psi_N \big\|_{(g)},
$$
then the points $\{a_{k}\}_{0\le k<N}$ are equispaced on the unit circle.
\end{theorem}

\begin{corollary}\label{cor:lb}
For every $\alpha\in (0,1]$, for every integer $N\ge 1$, and for every
family of points $\{a_{k}\}_{0\le k<N}$ on the unit circle we have
$$
\bigg\|\sum_{0\le k<N}\frac1{z-a_{k}}\bigg\|_{\alpha}\ge
\big\|\Psi_N \big\|_{\alpha}.
$$
\end{corollary}

It is easy to see that the sequence of norms
$$
 \big\|\Psi_N  \big\|_{\alpha}    =\bigg\| \frac{Nz^{N-1}}{z^N-1}\bigg\|_{\alpha}
$$
tends to zero as $N\to \infty$ for $0 < \alpha <1$, tends to a positive finite number for $\alpha= 1$, and tends to $+\infty$ for $\alpha >1$. The following result provides with the exact asymptotics.
As usual, we denote by $\zeta$ and $\Gamma$ the Riemann zeta-function and the Gamma function, respectively.

\begin{theorem}\label{pro:limit} For every $\alpha>0$ we have
$$
\lim_{N\to\infty}N^{\alpha-1}\|\Psi_N\|^2_{\alpha}=\Gamma(\alpha+2)\zeta(\alpha+1)>0.
$$
\end{theorem}

In particular,  
$$
\lim_{N\to\infty} \|\Psi_N\|_{1}=\frac{\pi}{\sqrt{3}}.
$$

For general $g$, we can obtain weaker asymptotical estimates on the norms $\|\Psi_N\|_{(g)}$.

\begin{proposition}\label{pro} Let $g$ satisfy \eqref{eq:g1}. Then   
$$
\|\Psi_N\|^2_{(g)}\asymp N\int_0^{1/N}\frac{g(t)\,dt}{t}+ N^2\int_{1/N}^1(1-t)^Ng(t)\,dt,\qquad N\to\infty.
$$
\end{proposition}

\begin{corollary}\label{coro} {\rm (A)} For every $c>0$ we have 
$$
\exp(-cN)\lesssim \|\Psi_N\|_{(g)}=o(N^{1/2}),\qquad N\to\infty.
$$

{\rm (B)} If $g(t)=o(t)$, $t\to 0$, then $\|\Psi_N\|_{(g)}=o(1)$, $N\to\infty$.

{\rm (C)} If $q>1$ and $g(t)=\log^{-q}(2/t)$, then 
$$
\|\Psi_N\|^2_{(g)}\asymp \frac{N}{\log^{q-1}N},\qquad N\to\infty.
$$

{\rm (D)} If $q>0$ and $g(t)=\exp(-t^{-q})$, then 
$$
\log(1/\|\Psi_N\|_{(g)})\asymp N^{q/(q+1)},\qquad N\to\infty.
$$
\end{corollary}


We do not know whether the equispaced distribution remains to be optimal for the spaces $A^2_{ \alpha}$ when $\alpha>1$. Nevertheless, we show that asymptotically this is true up to a constant:

\begin{theorem}\label{t4} Let $\alpha>1$. For some absolute constant $C_1>0$ and for some number $C_2(\alpha)>0$, we have
$$
\alpha C_1 N^{1-\alpha}\le \min_{a_k\in\mathbb T,\,0\le k<N} \bigg\|\sum_{0\le k<N}\frac1{z-a_k}\bigg\|^2_{\alpha}\le C_2(\alpha)N^{1-\alpha},\qquad N\ge 1.
$$
\end{theorem}


Given a weighted Bergman space, it is natural to ask which elements of the space can be approximated in norm by the  simplest fractions with poles on $\mathbb
T$.  Our next result answers this question. It turns out that one can approximate either ``everything" 
or ``nothing" depending on $g$: 

\begin{theorem}\label{thm:clos} Let $g\not\equiv 0$  
satisfy \eqref{eq:g1}. Then 
$$
\mathop{\rm clos}{}_{A^2_{(g)}}\SF=\begin{cases}
\SF, \qquad t=O(g(t)),\,t\to0, \\
A^2_{(g)},\qquad g(t)=o(t),\,t\to0.
\end{cases}
$$
\end{theorem}

In particular, $\SF$ is closed nowhere dense in $A^2_{\alpha}$ when $0<\alpha\le 1$ and is dense in $A^2_{\alpha}$ when $\alpha>1$. 

In the case $\alpha=1$ we have a more precise result. Set 
$$
\SF_N=\Bigl\{\sum_{0\le j<N}\frac1{z-z_{N,j}}:z_{N,j}\in\mathbb T,\,0\le j<N\Bigr\}.
$$
The sets $\SF_N$ are compact in $A^2_{\alpha}$ for $\alpha>0$, $N\ge 1$.

\begin{theorem}\label{thm8} For every $f\in A^2_1$, we have 
$$
\lim_{N\to\infty}\dist_{A^2_1}(f,\SF_N)=\frac{\pi}{\sqrt3}.
$$
\end{theorem}

This result shows, in particular, that the set $\SF$ is a $((\pi/\sqrt3) + \varepsilon)$-net
in the space $A^2_1$, for small $\varepsilon>0$; considering the functions $-\Psi_N$ with large $N$ 
we see that the set $\SF$ is not a 
$((\pi/\sqrt3) - \varepsilon)$-net in the space $A^2_1$, for small $\varepsilon>0$.

In 1967 Thompson \cite{MT} (answering a question posed by Korevaar in 1965) obtained that for every bounded analytic function $f$ in 
$\mathbb D$, there exist $h_n\in \bigcup_{N\ge n}\SF_N$, $n\ge 1$, converging to $f$ uniformly on compact subsets 
of $\mathbb D$ and such that 
$$
\sup_{n\ge 1,\,z\in\mathbb D}(1-|z|)|h_n(z)|<\infty.
$$
His proof used the results and the constructions by Mac Lane in \cite{ML}. 
Let us formulate a somewhat improved version of Thompson's theorem. 

Let $H^\infty=H^\infty(\mathbb D)$ denote the space of bounded 
analytic functions in the unit disc. 

\begin{theorem}\label{thm9} Let $f\in H^\infty$. For every $\varepsilon>0$, for every compact subset $K$ of $\mathbb D$,  and 
for every $N\ge N(f,\varepsilon,K)$ there exists $h\in \SF_N$ such that 
\begin{gather*}
\|f-h\|_{L^\infty(K)}\le\varepsilon,\\
|h(z)|\le \frac1{1-|z|}+C_0\|f\|_{H^\infty}\log\frac e{1-|z|},\qquad z\in\mathbb D,
\end{gather*}
for some absolute constant $C_0$.
\end{theorem}



To prove Theorems~\ref{thm:clos} and \ref{thm8} we use an $L^p$ version of Thompson's theorem which we will formulate below. Whereas the simplest fractions $h$ constructed in the proof of Theorem~\ref{thm9} have ``almost'' equispaced poles, our Theorem~\ref{thm10} shows that the average growth of $h$ along the concentric circles $r\mathbb T$ is not much 
faster than that of the corresponding simplest fraction $\Psi_N$.

Given $\beta>0$, denote
$$
\rho(\beta)=\frac{1+\beta}{((1+\beta)^{1/(p-1)}-1)^{p-1}}>0 
$$
for $1< p<\infty$ and $\rho(\beta)=1$ for $p=1$, 
so that, by a simple calculation, we have
$$
(x+y)^p\le (1+\beta)x^p+\rho(\beta)y^p,\qquad x,y\ge 0.
$$

\begin{theorem}\label{thm10} Let $f\in H^\infty$, $1\le p<\infty$. For every $\varepsilon,\beta>0$, for every compact subset $K$ of $\mathbb D$, and 
for every $N\ge N(f,\varepsilon,K)$ there exists $h\in \SF_N$ such that 
\begin{align}
\|f-h\|_{L^\infty(K)}&\le\varepsilon,\notag\\
\label{drff}\int_0^1|h(e^{2\pi i s}r)|^p\,ds &\le (1+\beta)\int_0^1|\Psi_N(e^{2\pi i s}r)|^p\,ds\\
&\quad+\rho(\beta)C^p_0\|f\|^p_{H^\infty}\log^p\frac e{1-r},\qquad 0<r<1,\notag
\end{align}
for $C_0$ as in Theorem~\ref{thm9}.
\end{theorem}

\begin{remark}\label{r125} Our estimates on $h$ in Theorem~\ref{thm10} improve on those in Theorem~\ref{thm9}. 
Namely, for $1<p<\infty$ and $r\in(1-N^{-1},1)$, Theorem~\ref{thm10} gives 
$$
I_r:=\int_0^1|h(e^{2\pi i s}r)|^p\,ds\lesssim N(1-r)^{1-p},
$$ 
which improves 
on the estimate $I_r\lesssim (1-r)^{-p}$ that we can get from Theorem~\ref{thm9}.
If $1-r=A/N$ for large fixed $A$, then Theorem~\ref{thm10} gives $I_r\lesssim N^pe^{-pA}$ while Theorem~\ref{thm9} gives 
$I_r\lesssim N^pA^{-p}$, $N\to\infty$.
\end{remark}

Our results motivate the following open questions.

\begin{question}
Does Theorem~\ref{thm:lb} hold for larger classes of $g$? For instance, for $g(t)=t^\alpha$, $\alpha>1$\,?
\end{question}

\begin{question} It would be of interest to have more information about the mutual location of the sets 
$\SF_n$ in the spaces $A^2_\alpha$.
In particular, are pairwise distances between these sets bounded away from $0$ in the space $A^2_1$\,? 
Our conjecture is that the answer is positive, and, moreover,
if $\alpha>0$, $n,k\ge 1$, then
$$
\dist_{A^2_\alpha}(\SF_n,\SF_{n+k})=\|\Psi_k\|_{\alpha}\,.
$$
\end{question}
 
We finish this section with a few words about the organization of the paper and the methods used.
 
Theorems~\ref{thm:lb}, \ref{pro:limit}, and \ref{t4} and Proposition~\ref{pro} are proved  in Section~\ref{S4}. The proof of Theorem~\ref{thm:lb} uses some classical results on trigonometric series and a convexity argument, which is discussed  in Section~\ref{au}. The proofs of Theorem~\ref{pro:limit} and Proposition~\ref{pro} are direct calculations. In  the proof of Theorem~\ref{t4} we use moment estimates for systems of unimodular numbers 
going back to J.~W.~S.~Cassels. 

In Section~\ref{S5} we establish Theorems~\ref{thm9} and \ref{thm10} generalizing Thompson's theorem. 
Using the ideas of \cite{ML} and \cite{MT}, we provide a short argument with better pointwise and integral estimates.

Theorems~\ref{thm:clos} and \ref{thm8} are proved in Section~\ref{S6}. Their proofs use Theorem~\ref{thm10}. 
In Remark~\ref{r13} we indicate an alternative way to get the density of $\SF$ in $A^2_\alpha$, $\alpha>1$.

\section{Auxiliary lemmas}
\label{au}

Let $g$ be a function satisfying the conditions of Theorem~\ref{thm:lb}.
For integer $k\ge 0$ we set
$$
c_{(g),k}=\int_0^1 t^kg(1-t)\,dt>0,
$$
and define the function
$$
\varphi_{(g)}(t)=\sum_{k\ge 0}c_{(g),k}\cos((k+1)t), \qquad t\in\mathbb R.
$$
Notice that condition \eqref{eq:g1} is equivalent to the fact that
$\varphi_{(g)}(0)<\infty$.

Next, for every $\alpha>0$, let $g_\alpha(t)=t^\alpha$, $t\ge 0$, $c_{\alpha,k}=c_{(g_\alpha),k}$, and $\varphi_\alpha=\varphi_{g_\alpha}$, so that 
$$
c_{\alpha,k}=\int_0^1 t^k\,(1-t)^{\alpha}\,dt,\qquad k\ge 0
$$
and
$$
\varphi_{\alpha}(t)=\sum_{k\ge 0}c_{\alpha,k}\cos((k+1)t), \qquad t\in\mathbb R.
$$
Notice that for $\alpha>0$ we have 
$$
c_{\alpha,k}\asymp k^{-(\alpha+1)},\qquad k\to\infty.
$$

Both $\varphi_{(g)}$ (for the aforesaid $g$) and
$\varphi_{\alpha}$ (for $\alpha>0$) are $2\pi$-periodic even continuous functions.

We need the following convexity lemma.

\begin{lemma}\label{lem:convexity}\hfill\par
{\rm (1)} For every function $g$ satisfying the conditions of Theorem~\ref{thm:lb}, the function
$\varphi_{(g)}$ is strictly convex on $(0,2\pi)$.

\smallskip
{\rm (2)} The function $\varphi_{\alpha}$, $\alpha>0$, is strictly convex on
$(0,2\pi)$ if and only if $\alpha\in(0,1]$.
\end{lemma}

\begin{proof} 
(A) First we prove that the function $\varphi=\varphi_{1}$ is strictly convex. We have
$$
c_{1,k}=\frac1{(k+1)(k+2)},\qquad k\ge 0,
$$
and, hence,
$$
\varphi(t)=\sum_{k\ge 1}\frac{\cos(kt)}{k(k+1)}.
$$
Therefore, for every $t\in(0,2\pi)$ we obtain
\begin{multline*}
\varphi'(t)=-\sum_{k\ge 1}\frac{\sin(kt)}{k+1}=
-\sum_{k\ge 1}\frac{\sin(kt)}{k+1}+\sum_{k\ge 1}\frac{\sin(kt)}{k}
-\frac{\pi-t}2\\=
\sum_{k\ge 1}\frac{\sin(kt)}{k(k+1)}-\frac{\pi-t}2,
\end{multline*}
and, hence,
$$
\varphi''(t)=\frac12+\sum_{k\ge 1}\frac{\cos(kt)}{k+1}.
$$

Now we are going to use the following result from the book by N.~Bari \cite[Chapter~1, Section~30]{Bari1964book}.
Let $\{a_k\}_{k\ge 0}$ be a decreasing convex sequence of positive numbers, $\lim_{k\to\infty}a_k=0$. Then 
$(a_0/2)+\sum_{k\ge 1}a_k{\cos(kt)}\ge 0$, $t\in(0,2\pi)$, because
$$
\frac{a_0}2+\sum_{k\ge 1}a_k{\cos(kt)}=\frac12\sum_{j\ge 0}(j+1)\Delta^2a_jF_{j+1}(t),\qquad t\in(0,2\pi),
$$
where $\Delta^2a_j=\Delta a_j-\Delta a_{j+1}$, $\Delta a_j=a_j-a_{j+1}$, $j\ge 0$, and $F_j$ are the Fej\'er kernels,
$$
F_j(t)=\frac1j\Bigl(\frac{\sin(jt/2)}{\sin(t/2)}\Bigr)^2\ge 0,\qquad j\ge 1.
$$
In our situation, $a_k=1/(k+1)$, $\Delta^2a_k>0$, $k\ge 0$, and, hence, we have
$\varphi''(t)> 0$ on $(0,2\pi)$.

(B) Let $g$ be a function satisfying the conditions of Theorem~\ref{thm:lb}. Then we have
\begin{align*}
\varphi_{(g)}(t)&=\sum_{k\ge 0}\cos((k+1)t)\int_0^1s^k\,g(1-s)\,ds\\&=
\Re\, e^{it}\,\int_0^1\sum_{k\ge 0}e^{itk}\,s^k\,g(1-s)\,ds\\&=
\Re{\int_0^1\frac1{e^{-it}-s}\,g(1-s)\,ds}\\&=
\int_0^1\frac{\cos{t}-s}{1+s^2-2s\cos{t}}\,g(1-s)\,ds. 
\end{align*}
Hence, $\varphi_{(g)}\in C^{\infty}((0,2\pi))$. 

Furthermore,
\begin{align*}
\varphi_{(g)}(t)&=
\int_0^1\frac{\cos{t}-s}{1+s^2-2s\cos{t}}\,g(1-s)\,ds\\&=
-\frac12\int_0^1g(1-s)\,d\log(1+s^2-2s\cos{t})\\&=
-\frac12\int_0^1\log(1+s^2-2s\cos{t})\,g'(1-s)\,ds,
\end{align*}
for $0<t<2\pi$. Hence,
$$
\varphi_{(g)}'(t)=-\int_0^1\frac{s\sin{t}}{1+s^2-2s\cos{t}}\,g'(1-s)\,ds,
\qquad 0<t<2\pi,
$$
and
\begin{multline*}
\varphi_{(g)}''(t)=
-\int_0^1\frac{s\cos{t}(1+s^2-2s\cos{t})-2s^2\sin^2{t}}{(1+s^2-2s\cos{t})^2}\,g'(1-s)\,ds\\=
\int_0^1s\frac{2s-(1+s^2)\cos{t}}{(1+s^2-2s\cos{t})^2}\,g'(1-s)\,ds,\qquad
0<t<2\pi.
\end{multline*}
Thus, $\varphi_{(g)}$ is strictly convex on $\big[\frac\pi2,\pi\big]$. Next, let us
observe that $\varphi_{(g)}''(t)=\varphi_{(g)}''(2\pi-t)$ on $(0,2\pi)$. Hence, $\varphi_{(g)}$ is strictly convex on 
$\big[\frac\pi2,\frac{3\pi}2\big]$.

Since the function $\varphi_1$ is strictly convex on $(0,2\pi)$, we obtain that
$$
\int_0^1 h_t(s)\,ds> 0,
$$
for $0<t<\dfrac\pi2$, where
$$
h_t(s)=s\frac{2s-(1+s^2)\cos{t}}{(1+s^2-2s\cos{t})^2}.
$$
Take now $t\in\big(0,\frac\pi2\big)$ and choose (the unique) $s_t\in(0,1)$
such that $2s_t=(1+s_t^2)\cos{t}$. Then $h_t(s)<0$ for $s\in(0,s_t)$ and
$h_t(s)>0$ for $s\in(s_t,1)$. Furthermore,
\begin{multline*}
\varphi_{(g)}''(t)=\int_0^1h_t(s)\,g'(1-s)\,ds\\ \ge
\int_0^{s_t}h_t(s)\,g'(1-s_t)\,ds+
\int_{s_t}^1h_t(s)\,g'(1-s_t)\,ds=\\
g'(1-s_t)\int_0^1h_t(s)\,ds\ge 0.
\end{multline*}

Suppose now that $\varphi_{(g)}''(t)=0$. Then $g'(1-s_t)=0$, and, hence, $g'=0$ on the interval $[1-s_t,1]$, and  
$\varphi_{(g)}''(t)=\int_{s_t}^1h_t(s)\,g'(1-s)\,ds>0$, which is impossible.
Therefore, the function $\varphi_{(g)}$ is strictly convex on $(0,\frac\pi2)$. Using once again that 
$\varphi_{(g)}''(t)=\varphi_{(g)}''(2\pi-t)$, we conclude that $\varphi_{(g)}$ is strictly convex on $(0,2\pi)$.

(C) By the result of (B), the function $\varphi_{\alpha}$ is strictly convex on $(0,2\pi)$ for $\alpha\in(0,1]$.

(D) It remains to notice that for any $\alpha>1$ the function
$\varphi_{\alpha}$ is not convex. Indeed, for such $\alpha$ we have
$\varphi_{\alpha}\in C^1(\mathbb R)$, and, since $\varphi_{\alpha}$ attains its
maximum at the point $t=0$, this function cannot be convex on $(0,2\pi)$. 
The
lemma is proved.
\end{proof}


The next lemma pertains to the convex analysis.

\begin{lemma}\label{lem1}
Let $\varphi$ be a $2\pi$-periodic even continuous function strictly convex on $(0,2\pi)$. Then for every $N\ge 2$ we have
\begin{equation}
\inf_{\vartheta_j\in[0,2\pi),\,0\le j< N}\sum_{0\le  j,k<N,\,j\neq k}
\varphi(\vartheta_j-\vartheta_k)=
\sum_{0\le  j,k<N,\,j\neq k}
\varphi\big(\tfrac{2\pi j}{N}-\tfrac{2\pi k}{N}\big).
\label{d2}
\end{equation}
Furthermore, if $\vartheta_j\in[0,2\pi)$, $0\le j< N$, and 
$$
\sum_{0\le  j,k<N,\,j\neq k}
\varphi(\vartheta_j-\vartheta_k)=
\sum_{0\le  j,k<N,\,j\neq k}
\varphi\big(\tfrac{2\pi j}{N}-\tfrac{2\pi k}{N}\big),
$$
then the points $e^{i\vartheta_j}$ are equispaced on the unit circle.
\end{lemma}

\begin{proof} 
Changing, if necessary, the enumeration of $\vartheta_j$, we can assume that the infimum in \eqref{d2} is taken
over $0\le \vartheta_0\le  \vartheta_1\le\ldots\le  \vartheta_{N-1}\le 2\pi$. Furthermore, set $\vartheta_{k+jN}=\vartheta_k+2\pi j$, $0\le k<N$, $j\in\mathbb Z$.
Since
$$
\sum_{0\le  j,k<N,\,j\neq k}
\varphi(\vartheta_j-\vartheta_k)=\sum_{1\le  s<N}\sum_{0\le  j<N}
\varphi(\vartheta_{j+s}-\vartheta_j),
$$
it suffices to verify that for every $1\le  s<N$
$$
\sum_{0\le  j<N}\varphi(\vartheta_{j+s}-\vartheta_j)\ge N\varphi\big(\tfrac{2\pi s}{N}\big),
$$
and that the equality is attained only if the differences $\vartheta_{j+s}-\vartheta_j$ do not depend on $j$.

Fix $1\le s<N$ and set $n=\gcd(N,s)$, $m=N/n$. It remains to prove 
\begin{equation}
\sum_{0\le  p<m}\varphi(\vartheta_{j+(p+1)s}-\vartheta_{j+ps})\ge m\varphi\big(\tfrac{2\pi s}{N}\big),\qquad 0\le  j<n,
\label{d3q}
\end{equation}
and 
\begin{equation}
\left\{\begin{gathered}\text{the equality in \eqref{d3q} is attained only if} \\
\text{\,$\vartheta_{j+(p+1)s}-\vartheta_{j+ps}=\tfrac{2\pi s}{N}$, $
0\le j<n$, $0\le p<m$}.
\end{gathered}\right.
\label{d3q9}
\end{equation}

Denote $\ell=s/n$. We have $1\le \ell<m$. Let
$$
\mathcal X_\ell=\Bigl\{ x=(x_0,\ldots,x_{m-1})\in[0,2\pi]^m:\sum_{0\le q<m}x_q=2\pi \ell \Bigr\}.
$$
Now, to get \eqref{d3q} and \eqref{d3q9} we need only to establish the inequality
\begin{equation}
\sum_{0\le  q<m}\varphi(x_q) \ge m\varphi\big(\tfrac{2\pi \ell}{m}\big),\qquad x\in\mathcal X_\ell,
\label{d3qq}
\end{equation}
and the property
\begin{equation}
\sum_{0\le  q<m}\varphi(x_q) = m\varphi\big(\tfrac{2\pi \ell}{m}\big)\iff 
x=(\tfrac{2\pi \ell}{m},\tfrac{2\pi \ell}{m},\ldots,\tfrac{2\pi \ell}{m}), \quad x\in\mathcal X_\ell.\label{d3qq1}
\end{equation}

By a compactness argument, we can find $y=(y_0,\ldots,y_{m-1})\in\mathcal X_\ell$ such that
$$
\sum_{0\le  q<m}\varphi(y_q) =\min_{x\in\mathcal X_\ell}\sum_{0\le  q<m}\varphi(x_q).
$$
%
If $y_q\not=y_{q+1}$ for some $q$, then we can replace $y_q$ and $y_{q+1}$ by $(y_q+y_{q+1})/2$, and, 
by the strict convexity of $\varphi$, the sum of the values of $\varphi$ will decay, which is impossible.
Hence, $y_q=2\pi \ell/m$, $0\le q<m$, which proves \eqref{d3qq} and \eqref{d3qq1}, and concludes the proof of the lemma.
\end{proof}

\section{Proofs of Theorems~\ref{thm:lb}, \ref{pro:limit}, \ref{t4}, and Proposition~\ref{pro}}
\label{S4}


\begin{proof}[Proof of Theorem~\ref{thm:lb}]
Let $N\ge 1$, $\vartheta_k\in[0,2\pi]$, $0\le k<N$,
$$
f(z)=\sum_{0\le k<N}\frac1{z-e^{i\vartheta_k}},\qquad z\in\mathbb D.
$$
Denote by $\|\cdot\|$ the norm and by $\langle\cdot,\cdot\rangle$ the scalar product in the space $A^2_{(g)}$.
If $\vartheta=\vartheta_2-\vartheta_1$, then
\begin{align*}
&\Re\bigg\langle\frac1{z-e^{i\vartheta_1}},\frac1{z-e^{i\vartheta_2}}\bigg\rangle\\&=
\kappa_g\Re\int_{\mathbb D}\bigg(\sum_{k\ge 0}e^{-i(k+1)\vartheta_1}z^k\bigg)\cdot
\bigg(\overline{\sum_{j\ge 0}e^{-i(j+1)\vartheta_2}z^j}\bigg)
\cdot g(1-|z|^2)\,dm_2(z)\\
&=\kappa_g\Re \int_{\mathbb D}\sum_{k\ge 0}|z|^{2k}e^{i\vartheta(k+1)}
g(1-|z|^2)\,dm_2(z)\\
&=2\kappa_g\Re\int_0^1\sum_{k\ge 0}r^{2k}e^{i\vartheta(k+1)}
g(1-r^2)\,r\,dr\\
&=\kappa_g\Re\displaystyle\sum_{k\ge 0}e^{i\vartheta(k+1)}
\displaystyle\int_0^1t^kg(1-t)\,dt\\&=\kappa_g\sum_{k\ge 0}c_{(g),k}\cos((k+1)\vartheta)
=\kappa_g\varphi_{(g)}(\vartheta).
\end{align*}

Therefore,
\begin{align*}
\|f\|^2&=\bigg\langle\sum_{0\le k<N}\frac1{z-e^{i\vartheta_k}},
\sum_{0\le k<N}\frac1{z-e^{i\vartheta_j}}\bigg\rangle
\\&=\sum_{0\le k<N}\Bigl\|\frac1{z-e^{i\vartheta_k}}\Bigr\|^2+
\kappa_g\sum_{0\le j,k<N,\,j\ne k}\varphi_{(g)}(\vartheta_k-\vartheta_j)\\&=N\Bigl\|\frac1{z-1}\Bigr\|^2+
\kappa_g\sum_{0\le j,k<N,\,j\ne k}\varphi_{(g)}(\vartheta_k-\vartheta_j).
\end{align*}
By Lemma~\ref{lem:convexity}, 
the function $\varphi_{(g)}$ is strictly convex on $(0,2\pi)$, and by Lemma~\ref{lem1},
the quantity $\|f\|$ attains its minimum if and only if the points
$e^{i\vartheta_k}$, $0\le k<N$, are equispaced on the unit circle, 
and, hence, $\|f\|=\|\Psi_N\|$.
\end{proof}

\begin{proof}[Proof of Theorem~\ref{pro:limit}]
Given $N\ge 1$, we have
\begin{align*}
\|\Psi_N\|_{\alpha}^2&=(\alpha+1)\int_{\mathbb D}
\bigg|\frac{Nz^{N-1}}{1-z^N}\bigg|^2(1-|z|^2)^\alpha\,dm_2(z)\\&=\frac{\alpha+1}{\pi}\int_0^1\biggl(\int_0^{2\pi}\frac{dt}{|1-r^N e^{iNt}|^2}\biggr)
(Nr^{N-1})^2(1-r^2)^\alpha r\,dr.
\end{align*}
By a direct computation one verifies that
$$
\int_0^{2\pi}\frac{dt}{|1-x e^{it}|^2}=\frac{2\pi}{1-x^2},\qquad 0\le x <1.
$$
Therefore,
$$
\|\Psi_N\|_{\alpha}^2=2(\alpha+1)\int_0^1
\frac{N^2r^{2N-2}(1-r^2)^\alpha r}{1-r^{2N}}\,dr
$$
Using the substitution $r=e^{-s/(2N)}$, we obtain
$$
N^{\alpha-1}\|\Psi_N\|_{\alpha}^2=(\alpha+1)\int_0^{+\infty}
\frac{\bigl(N(1-e^{-s/N})\bigr)^\alpha}{e^s-1}\,ds.
$$
Since $1-e^{-x}\le  x$ for $x\ge 0$, by the Lebesgue dominated convergence theorem we conclude that
\begin{align*}
\lim_{N\to\infty}N^{\alpha-1}\|\Psi_N\|^2_{\alpha}&=
(\alpha+1)\int_0^{+\infty}
\frac{s^\alpha}{e^s-1}\,ds\\&=(\alpha+1)\sum_{k\ge 1}\int_0^{+\infty}
s^\alpha e^{-ks}\,ds\\&=(\alpha+1)\sum_{k\ge 1}k^{-\alpha-1}\int_0^{+\infty}
s^\alpha e^{-s}\,ds\\&=\Gamma(\alpha+2)\zeta(\alpha+1).
\end{align*}
\end{proof}

\begin{remark}\label{r12}
The same calculation shows that for every $\alpha>0$, the sequence $\{N^{\alpha-1}\|\Psi_N\|_{\alpha}^2\}_{N\ge 1}$ is 
monotonically increasing.
\end{remark}

\begin{proof}[Proof of Proposition~\ref{pro}]
As in the proof of Theorem~\ref{pro:limit}, we have
\begin{align*}
\|\Psi_N\|&_{(g)}^2\asymp\int_0^1
\frac{N^2r^{2N-2}g(1-r^2)r}{1-r^{2N}}\,dr\asymp 
N^2\int_0^1
\frac{r^{N-1}g(1-r)}{1-r^N}\,dr\\&=
N^2\int_0^{1-(1/N)}
\frac{r^{N-1}g(1-r)}{1-r^N}\,dr+
N^2\int_{1-(1/N)}^1
\frac{r^{N-1}g(1-r)}{1-r^N}\,dr\\ &\asymp
N^2\int_0^{1-(1/N)} r^Ng(1-r)\,dr+
N^2\int_{1-(1/N)}^1\frac{g(1-r)}{N\cdot(1-r)}\,dr\\&=
N^2\int_{1/N}^{1} (1-r)^Ng(r)\,dr+
N\int_0^{1/N}\frac{g(r)}{r}\,dr.
\end{align*}
\end{proof}

\begin{proof}[Proof of Theorem~\ref{t4}]
The upper estimate follows from Theorem~\ref{pro:limit}. 

Fix $\alpha>1$, $N\ge 1$, and $a_k\in\mathbb T$, $0\le k<N$. To establish the lower estimate,
it suffices to verify that for some absolute constant $C>0$ we have  
\begin{equation}
I:=\int_{N^{-1}<1-|z|^2<2N^{-1}}
\bigg|\sum_{0\le k<N}\frac1{z-a_{k}}\bigg|^2\,dm_2(z)\ge CN.
\label{frt}
\end{equation}
Since
$$
\sum_{0\le k<N}\frac1{z-a_{k}}=-\sum_{s\ge 0}z^s\sum_{0\le k<N}\overline{a}_{k}^{s+1},
$$
we have
$$
I=\sum_{s\ge 0}\int_{N^{-1}<1-t<2N^{-1}}t^s
\Bigl|\sum_{0\le k<N}a_{k}^{s+1}\Bigr|^2\,dt\gtrsim N^{-1}\sum_{0\le s\le 2N-1} 
\Bigl|\sum_{0\le k<N}a_{k}^{s+1}\Bigr|^2.
$$

Now, to get \eqref{frt}, it remains to check that
\begin{equation}
\sum_{1\le j\le 2N}\Bigl|\sum_{0\le k<N}b_k^j\Bigr|^2\ge \delta N^2,
\label{d8}
\end{equation}
for some absolute constant $\delta>0$ and for every family of unimodular numbers $\{b_k\}_{0\le k<N}$.

Individually, the sums $S_j=\sum_{k=0}^{N-1}b_k^j$ could be of order $O(\sqrt N)$ for $1\le j\le N^B$, $B>1$, (see \cite{Erd} for a probabilistic approach and \cite{And} for a deterministic algebraic approach). However, the sum of the squares of the moduli of
$S_j$ for $j$ between $1$ and $(1+\varepsilon)N$ (not between $1$ and $N$) admits a good lower estimate like in \eqref{d8}. Our argument here is inspired by that of J.~W.~S.~Cassels in \cite{Cas}.

For every $M\ge 1$ we have
\begin{align*}
&\sum_{1\le j\le M}\Bigl(1-\frac{j}{M+1}\Bigr)\Bigl|\sum_{0\le k<N}b_k^j\Bigr|^2\\
&=
\sum_{1\le j\le M}\Bigl(1-\frac{j}{M+1}\Bigr)\Bigl(N+\sum_{0\le k,m<N,\,k\not=m} (b_k\bar b_m)^j\Bigr)\\&=
N\sum_{1\le j\le M}\Bigl(1-\frac{j}{M+1}\Bigr)+
\sum_{0\le k<m<N}\sum_{|j|\le M,\,j\not=0}\Bigl(1-\frac{|j|}{M+1}\Bigr)
(b_k\bar b_m)^j\\&=\frac{NM}2+\sum_{0\le k<m<N}\sum_{|j|\le M}\Bigl(1-\frac{|j|}{M+1}\Bigr)(b_k\bar b_m)^j-\sum_{0\le k<m<N}1\\ &\ge
\frac{NM}2+\sum_{0\le k<m<N}\sum_{|j|\le M}\Bigl(1-\frac{|j|}{M+1}\Bigr)(b_k\bar b_m)^j -\frac{N(N-1)}2\\&\ge \frac{N(M-N+1)}2,
\end{align*}
because the Fej\'er kernel is non-negative,
$$
\sum_{|j|\le M}\Bigl(1-\frac{|j|}{M+1}\Bigr)e^{ijx}=F_{M+1}(x)\ge 0,\qquad x\in\mathbb R.
$$
Choose now $M=2N$. Then
$$
\sum_{1\le j\le 2N}\Bigl|\sum_{0\le k<N}b_k^j\Bigr|^2\ge
\sum_{1\le j\le 2N}\Bigl(1-\frac{j}{2N+1}\Bigr)\Bigl|\sum_{0\le k<N}b_k^j\Bigr|^2\ge \frac{N^2}2.
$$
\end{proof}



\section{Proofs of Theorems~\ref{thm9} and \ref{thm10}}
\label{S5}

\begin{proof}[Proof of Theorem~\ref{thm9}] Denote $M=\|f\|_{H^\infty}$. 
Given $N\ge 1$, set
$$
W_N(t)=Nt-2\int_0^t\Re(e^{2\pi iu}f(e^{2\pi iu}))\,du,\qquad t\ge 0.
$$
We have $W_N(0)=0$, $W_N(1)=N$,
$$
|W'_N(t)-N|\le 2M,\qquad t\ge 0.
$$
For sufficiently large $N$, the function $W_N$ increases, and we set $x_{N,k}=W_N^{-1}(k)$, $0\le k\le N$. 
We have 
\begin{equation}
|x_{N,k+1}-x_{N,k}|=N^{-1}+O(N^{-2}M), \qquad 0\le k<N, \,N\to\infty.
\label{dopp1}
\end{equation}
Put
$$
h_N(z)=\sum_{0\le k<N}\frac1{z-e^{2\pi ix_{N,k}}}.
$$
Then $h_N\in\SF_N$. 

Given $z=re^{2\pi is}\in\mathbb D$, choose $0\le m<N$ such that 
$$
\bigl|e^{2\pi is}-e^{2\pi ix_{N,m}}\bigr|=\min_{0\le k<N}\bigl|e^{2\pi is}-e^{2\pi ix_{N,k}}\bigr|.
$$
Set $y_{N,k}=x_{N,m}+(k-m)/N$, $0\le k< N$. By \eqref{dopp1}, we have 
$$
|e^{2\pi ix_{N,k}}-e^{2\pi iy_{N,k}}|=O(N^{-1}M|1-e^{2\pi i (k-m)/N}|), \quad 0\le k<N, \,N\to\infty, 
$$
and
$$
|z-e^{2\pi ix_{N,k}}|\asymp|z-e^{2\pi iy_{N,k}}|, \qquad 0\le k<N, \,N>N(M).
$$
Furthermore, 
$$
\sum_{0\le k<N}\frac1{z-e^{2\pi iy_{N,k}}}=e^{-2\pi ix_{N,m}}\Psi_N(ze^{-2\pi ix_{N,m}}).
$$
Therefore, 
\begin{align}
\label{New}&\bigl|h_N(z)-e^{-2\pi ix_{N,m}}\Psi_N(ze^{-2\pi ix_{N,m}})\bigr|  \\
&\le \sum_{0\le k<N}\Bigl| \frac1{z-e^{2\pi ix_{N,k}}}-\frac1{z-e^{2\pi iy_{N,k}}}\Bigr| \notag\\
&=\sum_{0\le k<N} \frac{|e^{2\pi ix_{N,k}}-e^{2\pi iy_{N,k}}|}{|z-e^{2\pi ix_{N,k}}|\cdot|z-e^{2\pi iy_{N,k}}|} \notag\\
&=O(N^{-1}M)\sum_{0\le k<N} \frac{|1-e^{2\pi i (k-m)/N}|}{|z-e^{2\pi iy_{N,k}}|^2} \notag\\
&=O(N^{-1}M)\Bigl(\sum_{0\le j<N(1-|z|)}  \frac{j/N}{(1-|z|)^2} +\sum_{N(1-|z|)\le j<N} \frac{j/N}{(j/N)^2} \Bigr) \notag\\ 
&=O(N^{-1}M)\Bigl(\frac{1}{N(1-|z|)^2}\!\!\sum_{0\le j<N(1-|z|)}  j \,+\,N\!\!\!\!\sum_{N(1-|z|)\le j<N} \frac{1}{j} \Bigr) \notag\\ 
&\le 
C_0M\log\frac e{1-|z|}, \notag
\end{align}
for some absolute constant $C_0$, for $N\ge N(M)$.

Since
$$
|\Psi_N(z)|  \le \frac1{1-|z|},
$$
we conclude that
\begin{align*}
|h_N(z)|&\le \bigl|e^{-2\pi ix_{N,m}}\Psi_N(ze^{-2\pi ix_{N,m}})\bigr|\\
&\qquad\qquad\qquad\qquad+\bigl|h_N(z)-e^{-2\pi ix_{N,m}}\Psi_N(ze^{-2\pi ix_{N,m}})\bigr|
\\ &\le \frac1{1-|z|}+C_0M\log\frac e{1-|z|},\qquad z\in\mathbb D,\,N\ge N(M).
\end{align*}

Furthermore, given $z\in \mathbb D$, we have
\begin{multline*}
f(z)=\frac1{2\pi i}\int_{\mathbb T}\frac{f(\zeta)}{\zeta-z}\,d\zeta=\frac1{2\pi i}\int_{\mathbb T}
\frac{\zeta f(\zeta)+\overline{\zeta f(\zeta)}}{\zeta(\zeta-z)}\,d\zeta\\=\frac1{\pi i}\int_{\mathbb T}
\frac{\Re(\zeta f(\zeta))}{\zeta(\zeta-z)}\,d\zeta=2\int_0^1
\frac{\Re(e^{2\pi it} f(e^{2\pi it}))}{e^{2\pi it}-z}\,dt
\end{multline*}
and
$$
0=\int_0^1\frac{1}{e^{2\pi it}-z}\,dt.
$$
Hence,
$$
f(z)=\int_0^1\frac{W'_N(t)}{z-e^{2\pi it}}\,dt,
$$
and
\begin{align*}
|f(z)-&h_N(z)|= \biggl|\sum_{0\le k<N} \biggl(\int_{x_{N,k}}^{x_{N,k+1}} 
\frac{W'_N(t)}{z-e^{2\pi it}}\,dt-\frac1{z-e^{2\pi ix_{N,k}}}\biggr)\biggr|\\
&=\biggl| \sum_{0\le k<N}\int_{x_{N,k}}^{x_{N,k+1}} \frac{e^{2\pi it}-e^{2\pi ix_{N,k}}}{(z-e^{2\pi it})(z-e^{2\pi ix_{N,k}})}W'_N(t)\,dt\biggr|\\
&=2\pi\biggl|
\sum_{0\le k<N} \int_{x_{N,k}}^{x_{N,k+1}} e^{2\pi ix_{N,k}}\frac{t-x_{N,k}+O(N^{-2})}{(z-e^{2\pi ix_{N,k}})^2+O(N^{-1})}W'_N(t)\,dt\biggr|\\
&= \frac{\pi}N\biggl|
\sum_{0\le k<N} \Bigl(\frac{e^{2\pi ix_{N,k}}}{(z-e^{2\pi ix_{N,k}})^2}+O(N^{-1}M)\Bigr)\biggr|
\end{align*}
for $N\ge N(M)$. 
Since $K$ is a compact subset of $\mathbb D$, we have
$$
\Bigl|\sum_{0\le k<N}\frac{e^{2\pi iy_{N,k}}}{(z-e^{2\pi iy_{N,k}})^2}\Bigr|=
\Bigl|\frac{N^2z^{N-1}}{(e^{2\pi iNy_{N,1}}-z^N)^2}\Bigr|\to 0,
$$
uniformly in $z\in K$ as $N\to\infty$, 
and
\begin{gather*}
|f(z)-h_N(z)|\le\frac{\pi}N 
\sum_{0\le k<N} \Bigl|\frac{e^{2\pi ix_{N,k}}}{(z-e^{2\pi ix_{N,k}})^2}-\frac{e^{2\pi iy_{N,k}}}{(z-e^{2\pi iy_{N,k}})^2}\Bigr|+o(1)\\ \lesssim
\frac{C(K)M}{N^2}\sum_{0\le k<N} |1-e^{2\pi i (k-m)/N}|+o(1)=o(1), 
\end{gather*}
uniformly in $z\in K$ as $N\to\infty$, 
\end{proof} 

\begin{proof}[Proof of Theorem~\ref{thm10}] We use the notation from the proof of Theorem~\ref{thm9}. 
Given $f\in H^\infty$, $\varepsilon>0$, and a compact subset $K$ of $\mathbb D$, choose $N\ge N(f,\varepsilon,K)$ and  $h_N\in \SF_N$ constructed in the proof of Theorem~\ref{thm9} 
so that 
$$
\|f-h_N\|_{L^\infty(K)}\le\varepsilon,
$$

It remains to verify the integral estimate \eqref{drff}. Fix $r\in (0,1)$ and set
$$
v_1(e^{2\pi i t})=|\Psi_N(e^{2\pi i t}r)|,\qquad 0\le t<1.
$$
Since
$$
v_1(e^{2\pi i t}) =\frac{Nr^{N-1}}{|e^{2\pi i tN}r^N-1|},\qquad 0\le t<1,
$$
the function $t\mapsto v_1(e^{2\pi i t})$ is even and decreases on $[0,1/(2N)]$.
Furthermore, we set
\begin{align*}
v_2(e^{2\pi i t})&=\begin{cases}
v_1\bigl(e^{2\pi i t}\bigr),\qquad e^{2\pi i t}\in U:=\{e^{2\pi i u}: \  -\frac{\pi}N\le u\le \frac{\pi}N\},\\
v_1\bigl(e^{\pi i/N}\bigr),\qquad e^{2\pi i t}\notin U,
\end{cases}\\
w(t)&=\int_0^t v^p_2(e^{2\pi i u})\,du,\qquad 0\le t\le 1.
\end{align*}
Then the function $t\mapsto v_2(e^{2\pi i t})$ is decreasing on $[0,1]$, and 
the function $w$ is concave on $[0,1]$. 

If 
\begin{equation}
\bigl|e^{2\pi is}-e^{2\pi ix_{N,m}}\bigr|=\min_{0\le k<N}\bigl|e^{2\pi is}-e^{2\pi ix_{N,k}}\bigr|,
\label{cond87}
\end{equation}
then, by \eqref{New}, we have  
$$
|h_N(e^{2\pi i s}r)|\le|\Psi_N(ze^{-2\pi ix_{N,m}})|+ 
C_0M\log\frac e{1-r}.
$$
Using the argument in the proof of \eqref{New} we obtain that, under condition \eqref{cond87}, if 
$$
\bigl|e^{2\pi is}-e^{2\pi ix_{N,m}}\bigr|\ge \bigl|1-e^{\pi i/N}\bigr|,
$$
then 
$$
|h_N(e^{2\pi i s}r)|\le|\Psi_N(e^{\pi i/N})|+ C_0M\log\frac e{1-r}.
$$
Since the function $t\mapsto v_1(e^{2\pi i t})$ is even, we conclude that 
$$
|h_N(e^{2\pi i s}r)|\le v_2(e^{2\pi i |s-x_{N,m}|})+ C_0M\log\frac e{1-r}.
$$

We divide the interval $[0,1]$ into subintervals 
$J_{2k}=[x_{N,k},(x_{N,k}+x_{N,k+1})/2]$, $J_{2k+1}=[(x_{N,k}+x_{N,k+1})/2,x_{N,k+1}]$, $0\le k< N$.

Then 
\begin{gather*}
\int_0^1|h_N(e^{2\pi i s}r)|^p\,ds =\sum_{0\le k<N}\int_{J_{2k}\cup J_{2k+1}}|h_N(e^{2\pi i s}r)|^p\,ds\\ \le
(1+\beta)\biggl(
\int_{J_{2N-1}\cup J_0}v^p_2(e^{2\pi i s})\,ds+\sum_{0< k<N}\int_{J_{2k-1}\cup J_{2k}}v^p_2(e^{2\pi i |s-x_{N,k}|})\,ds\biggr)\\+
\rho(\beta)C^p_0\|f\|^p_{H^\infty}\log^p\frac e{1-r}\\=
(1+\beta)\sum_{0\le k<N}\bigl( w(|J_{2k}|)+w(|J_{2k+1}|)\bigr)+
\rho(\beta)C^p_0\|f\|^p_{H^\infty}\log^p\frac e{1-r},
\end{gather*}
where $|J|$ is the length of $J$. 
Since the function $w$ is concave and $\sum_{0\le k<N} (|J_{2k}|+|J_{2k+1}|)=1$, we conclude that
\begin{gather*}
\int_0^1|h_N(e^{2\pi i s}r)|^p\,ds \le 
(1+\beta)\cdot 2Nw\Bigl(\frac1{2N}\Bigr)+
\rho(\beta)C^p_0\|f\|^p_{H^\infty}\log^p\frac e{1-r}\\=(1+\beta)\int_0^1|\Psi_N(e^{2\pi i s}r)|^p\,ds
+\rho(\beta)C^p_0\|f\|^p_{H^\infty}\log^p\frac e{1-r}.
\end{gather*}
\end{proof} 

\section{Proofs of Theorems~\ref{thm:clos} and \ref{thm8}}
\label{S6}

\begin{proof}[Proof of Theorem~\ref{thm:clos}]
Denote by $\mathcal S_{(g)}$ the closure of the set $\SF$ in $A^2_{(g)}$. 
Since $\|f\|_{(g_1)}\le\|f\|_{(g_2)}$ when $g_1\le g_2$, we need only to consider the cases $g(t)=t$ and 
$g(t)=o(t)$, $t\to 0$. 

(A) Let $g(t)=t$. Then $A^2_{(g)}=A^2_1$. 
We are going to verify that
\begin{equation}\label{eq:st}
\liminf_{N\to\infty}\inf_{g\in\SF_N}\|f-g\|^2_{1}\ge  \frac{\pi^2}3,\qquad f\in A^2_1.
\end{equation}
Since every $\SF_N$ is compact in $A^2_1$, we can then conclude that $\mathcal S_{(g)}=\SF$.

Assume that \eqref{eq:st} does not hold. Then for some $\varepsilon\in(0,\pi^2/12)$ we find $f\in A^2_1$, 
a sequence $\{N_m\}_{m\ge 1}$ such that $\lim_{m\to\infty}N_m=\infty$, and
a sequence $\{f_m\}_{m\ge 1}$, $f_m\in \SF_{N_m}$, $m\ge 1$, such that  
\begin{equation}\label{dd1}
\|f-f_m\|^2_{1}\le \frac{\pi^2}3-4\varepsilon,\qquad m\ge 1.
\end{equation}

Given $\delta\in(0,1)$, put $g_\delta(t)=\min(\delta,t)$. Since
\begin{equation}
|f(z)|^2g_\delta(1-|z|^2)\le|f(z)|^2(1-|z|^2), \qquad z\in\mathbb D,
\label{dopp2}
\end{equation}
and the function $z\mapsto |f(z)|^2(1-|z|^2)$ is integrable on $\mathbb D$, by Lebesgue's 
dominated convergence theorem we have
$$
\lim_{\delta\to 0}\int_{\mathbb D}|f(z)|^2g_\delta(1-|z|^2)\,dm_2(z)=0.
$$
Choose $\delta\in(0,1)$ such that 
$$
\int_{\mathbb D}|f(z)|^2g_\delta(1-|z|^2)\,dm_2(z)\le \frac{\varepsilon^2}{8}.
$$
By \eqref{dd1} and \eqref{dopp2} we have 
$$
\int_{\mathbb D}|f(z)-f_m(z)|^2g_\delta(1-|z|^2)\,dm_2(z)\le \frac{\pi^2}6-2\varepsilon,\qquad m\ge 1,
$$
and, hence, 
\begin{align*}
\int_{\mathbb D}|f_m(z)|^2&g_\delta(1-|z|^2)\,dm_2(z)\\ &\le 
\Bigl(1+\frac{\varepsilon}{4}\Bigr)\int_{\mathbb D}|f(z)-f_m(z)|^2g_\delta(1-|z|^2)\,dm_2(z)\\&\qquad+
\Bigl(1+\frac4{\varepsilon}\Bigr)\int_{\mathbb D}|f(z)|^2g_\delta(1-|z|^2)\,dm_2(z)\\&\le 
\frac{\pi^2}6-\varepsilon,\qquad m\ge 1.
\end{align*}
Since the function $g_\delta$ is concave and non-decreasing, $g_\delta(0)=0$, and $\displaystyle\int_0
g_\delta(t)t^{-1}\,dt<\infty$, we can apply Theorem~\ref{thm:lb} and obtain
$$
\int_{\mathbb D}|\Psi_{N_m}(z)|^2g_\delta(1-|z|^2)\,dm_2(z)\le \frac{\pi^2}6-\varepsilon,\qquad m\ge 1.
$$
Since the functions $\Psi_{N_m}$ tend to $0$ uniformly on compact subsets of $\mathbb D$ as $m\to\infty$,  
we conclude that 
$$
\int_{\mathbb D}|\Psi_{N_m}(z)|^2(1-|z|^2)\,dm_2(z)\le \frac{\pi^2}6-\frac\varepsilon2,\qquad m\ge m(\delta),
$$
which contradicts to Theorem~\ref{pro:limit}.

This contradiction establishes relation \eqref{eq:st} and, hence, the equality $\mathcal S_{(g)}=\SF$ for $g(t)=t$.

(B) Let $g(t)=o(t)$, $t\to 0$, and let $f\in A^2_{(g)}$. 
Replacing $f$ by the function $z\mapsto f((1-\delta)z)$ with small positive $\delta$, we can assume that $f\in H^\infty$. 
By Theorem~\ref{thm10}, there exist $h_N\in\SF_{N}$, 
such that $h_N$ tend to $f$ uniformly on compact subsets 
of $\mathbb D$, $N\to\infty$, and for $r\in(0,1)$ we have
$$
\int_0^1|h_N(e^{2\pi i s}r)|^2\,ds \le 2\int_0^1|\Psi_{N}(e^{2\pi i s}r)|^2\,ds
+2C^2_0\|f\|^2_{H^\infty}\log^2\frac e{1-r}.
$$
By Corollary~\ref{coro}~(B), we conclude that 
$$
\|f-h_N\|^2_{(g)}\to 0,\qquad N\to\infty.
$$
Thus, $\mathcal S_{(g)}=A^2_{(g)}$.
\end{proof} 

\begin{remark}\label{r13} Given $\alpha>1$, denote by $\mathcal S_{\alpha}$ the closure of the set $\SF$ in $A^2_{\alpha}$.  
An alternative way to prove that $\mathcal S_{\alpha}=A^2_{\alpha}$ for $\alpha>1$, 
follows the scheme proposed by Korevaar in \cite{Kor1964ann} (with a reference to A.~Beurling). Namely, in the case $\alpha>2$, we use that $\lim_{N\to\infty}\|\Psi_N\|_\alpha=0$ to show that the functions $z\mapsto-(z-w)^{-1}$, $w\in\mathbb T$, 
belong to $\mathcal S_{\alpha}$. As a consequence, 
the real linear space $\mathcal R$ spanned by the family of
functions $\Bigl\{\dfrac{iw}{(z-w)^2}: w\in\mathbb T\Bigr\}$
is contained in the set $\mathcal S_{\alpha}$. In the case $1<\alpha\le 2$, using an explicit but more complicated argument 
we can establish that the real linear space
$\widetilde{\mathcal R}$ spanned by the family $\Bigl\{\dfrac1{z-w}:
w\in\mathbb T\Bigr\}$ is contained in the set $\mathcal S_{\alpha}$. 
Then an argument based on the Hahn--Banach
theorem permits us to show that, for $\alpha>2$, $\mathcal R$ is dense in 
$A^2_{\alpha}$ (for $\alpha>1$, $\widetilde{\mathcal R}$ is dense in 
$A^2_{\alpha}$), and to conclude.
\end{remark}

\begin{proof}[Proof of Theorem~\ref{thm8}] Because of \eqref{eq:st}, we need only to verify that for every 
$f\in A^2_1$, $\varepsilon>0$, and for every $N\ge N(f,\epsilon)$, there exists $h\in\SF_N$ such that  
$$
\|f-h\|^2_{1}\le  \frac{\pi^2}3+\varepsilon.
$$
Replacing $f$ by $z\mapsto f((1-\delta)z)$ with small positive $\delta$, we can assume that $f\in H^\infty$. 
Given $0<\beta<1$, choose $\eta\in (0,1)$ such that
$$
\int_{\mathbb D\setminus (1-\eta)\mathbb D}|f(z)|^2\,dm_2(z)+
\int_{1-\eta}^{1}(1-r^2)\log^2\frac e{1-r}\,dr<\beta^2.
$$ 
By Theorem~\ref{thm10}, for every $N\ge N(f,\epsilon)$, there exists $h\in\SF_N$ such that 
$$
\|f-h\|^2_{L^\infty((1-\eta)\mathbb D)}<\beta
$$
and 
\begin{multline*}
\int_0^1|h(e^{2\pi i s}r)|^2\,ds \le (1+\beta)\int_0^1|\Psi_N(e^{2\pi i s}r)|^2\,ds
\\+\frac{1+\beta}{\beta}C^2_0\|f\|^2_{H^\infty}\log^2\frac e{1-r},\qquad 0<r<1.
\end{multline*}
Then
\begin{align*}
\|f&-h\|^2_{1}\\&=2\int_{(1-\eta)\mathbb D}|f(z)-h(z)|^2(1-|z|^2)\,dm_2(z)
\\ &\qquad\qquad\qquad+
2\int_{\mathbb D\setminus (1-\eta)\mathbb D}|f(z)-h(z)|^2(1-|z|^2)\,dm_2(z)\\ &\le 
\beta+2(1+\beta^{-1})\int_{\mathbb D\setminus (1-\eta)\mathbb D}|f(z)|^2(1-|z|^2)\,dm_2(z)\\
&\qquad\qquad\qquad+2(1+\beta)\int_{\mathbb D\setminus (1-\eta)\mathbb D}|h(z)|^2(1-|z|^2)\,dm_2(z)\\ &\le
\beta+2\beta(1+\beta)+
2(1+\beta)^2\int_{\mathbb D\setminus (1-\eta)\mathbb D}|\Psi_N(z)|^2(1-|z|^2)\,dm_2(z)
\\
&\qquad\quad+
\frac{2(1+\beta)^2}{\beta}\int_{\mathbb D\setminus (1-\eta)\mathbb D}C^2_0\|f\|^2_{H^\infty}(1-|z|^2)\log^2\frac e{1-|z|}\,dm_2(z)
\\ &\le
(5+C_1\cdot C^2_0\|f\|^2_{H^\infty})\beta+2(1+\beta)^2\frac{\pi^2}{6}\le \frac{\pi^2}3+\varepsilon,
\end{align*}
for sufficiently small $\beta$.
\end{proof}

\end{document}